\documentclass[11pt]{article}

\usepackage[square,sort,comma,numbers]{natbib}
\usepackage{amsfonts}
\usepackage{amsmath,amssymb}
\usepackage{amscd}
\usepackage{amsthm}
\usepackage{fancyhdr}
\usepackage{color}
\usepackage{hyperref} 

\theoremstyle{plain}
\newtheorem{lem}{Lemma}[section]
\newtheorem{thm}{Theorem}
\newtheorem{clm}{Claim}

\newtheorem{cly}[lem]{Corollary} 
\newtheorem{prop}[lem]{Proposition}

\theoremstyle{definition}
\newtheorem{df}[lem]{Definition} 

\newtheorem{nota}[lem]{Notation} 

\theoremstyle{remark}
\newtheorem{rem}[lem]{Remark}

\newcommand{\ZZ}{\ensuremath{\mathbb Z}}

\newcommand{\RR}{\ensuremath{\mathbb R}}

\newcommand{\cL}{\mathcal{L}}
\newcommand{\cD}{\mathcal{D}}
\newcommand{\cF}{\mathcal{F}}

\newcommand{\cY}{\mathcal{Y}}
\newcommand{\ga}{\mathfrak{a}}
\newcommand{\g}{\mathfrak{g}}

\newcommand{\pd}[1]{\frac{\partial}{\partial #1}} 

\newcommand{\li}{\ensuremath{L_{\infty}}}

\title{Simultaneous deformations of algebras and morphisms via derived brackets}

\author{Ya\"el Fr\'egier\footnote{{  UArtois, LML, F-62\kern 1mm 300, Lens, France.}}\;\footnote{{MIT, 77 Mass. Avenue, Cambridge, MA 02139, USA.}}
\footnote{Universit\"at Z\"urich, Winterthurerstr. 190, CH-8057 Z\"urich, Switzerland. \texttt{yael.fregier@gmail.com}}
\and 
Marco Zambon\footnote{{ Universidad Aut\'onoma de Madrid }
{ (Departamento de Matem\'aticas) and}
{  ICMAT (CSIC-UAM-UC3M-UCM),}
{ Campus de Cantoblanco,}
{ 28049 Madrid, Spain. }
Email: 
\texttt{marco.zambon@uam.es,} \texttt{marco.zambon@icmat.es}. 
Current address: KU Leuven, Department of Mathematics, Celestijnenlaan 200B box 2400, BE-3001 Leuven, Belgium.
{\texttt{marco.zambon@wis.kuleuven.be}}
} 
 }

\begin{document}

\date{}

\maketitle
\begin{abstract}
We present a method to construct explicitly $L_\infty$-algebras governing simultaneous deformations of various kinds of algebraic structures and of their morphisms. It is an alternative to the heavy use of the operad machinery of the existing approaches.   Our method relies on Voronov's derived bracket construction.
\end{abstract}

\setcounter{tocdepth}{2} 

\section*{Introduction}
\addcontentsline{toc}{section}{Introduction}
 
The deformation theory of various kinds of structures (e.g. \cite{Ger}, \cite{KoSp1} and \cite{Ku}) can be encapsulated in the language of graded Lie algebras (\cite{NijRic1} and \cite{NijRic2}) or more generally, for non quadratic structures, of $L_\infty$-algebras (\cite{LadaStash}).   

 It is convenient to have such a formulation since cohomology theory, analogues of Massey products and a natural equivalence relation on the space of deformations come along for free. However obtaining such formulation -- that is, obtaining the $\li$- algebra governing the given deformation problem -- can be difficult. \\

There are known techniques (\cite{FMY}) 
   to solve this problem in the case of simultaneous deformations of {various kinds of } algebras and their morphisms, but they are based on the formalism of operads, which provides an obstacle 
to mathematicians not acquainted with operad theory. 

 On the other hand, T. Voronov, building on Y. Kosmann-Schwarzbach's derived brackets (\cite{yvette1}), developed techniques enabling to produce $L_\infty$-algebras out of some simple concepts of  graded linear algebra (\cite{Vor} and \cite{vor2}). In our work \cite{FZgeo} we showed how to adapt Voronov's results to the study of simultaneous deformations, and gave geometrical applications which could not be obtained otherwise.\\

In this paper we show that this approach also applies successfully to simultaneous deformations of algebras and morphisms and that this can be an alternative approach for users not willing to use the operadic formalism.\\


\noindent\textbf{Outline of the content of the paper.}   

\noindent In  \S\ref{sec:main} we recall the formalism of graded Lie algebras and $\li$-algebras together with the  \emph{derived bracket constructions} (see  Thm. \ref{voronovderived} and  \ref{voronov}) and {the tool we use to study simultaneous deformations}  (Thm. \ref{machine}). In \S \ref{Lie} we give algebraic applications to the study of simultaneous deformations of algebras and morphisms in the categories of Lie and $L_\infty$-algebras. Another application concerns Lie subalgebras of Lie algebras.

\section{$L_{\infty}$-algebras  via derived brackets and   Maurer-Cartan elements}\label{sec:main}

 {We recall the machinery we developed in \cite[\S1]{FZgeo} (which   {first appeared as} \cite[\S1]{YaelZ})}. The main result is Thm. \ref{machine}, which produces the $L_\infty$-algebras appearing in the rest of the article. We first give some basic material about $L_\infty$-algebras in \S \ref{homolie}, then we recall in \S \ref{vorpaper} Voronov's constructions which will be the main tools used to establish in \S \ref{heart} our Theorem \ref{machine}. We conclude justifying in \S\ref{sec:conv} why no convergence issues arise in our machinery, and discussing equivalences in \S\ref{sec:sym}.

{
We refer the reader to \cite[\S1]{FZgeo}  for additional  details and proofs (an exception being Lemma \ref{lem:eps}, which we prove here)}.
\subsection{Background on $L_{\infty}$-algebras}\label{homolie}

 {The notion of  $L_\infty$-algebra is  due to Lada and Stasheff \cite{LadaStash}, and contains graded Lie algebras and differential graded Lie algebras (DGLAs) as special cases. We will need only a ``shifted'' version of this notion, in which all the multibrackets are graded symmetric have degree one. We refer to the latter as $\li[1]$-algebras.}

 To introduce it, recall that given two elements $v_1, v_2$ in a graded vector space, the \emph{Koszul sign}   of the transposition $\tau_{1,2}$ of these two elements is $\epsilon(\tau_{1,2},v_1,v_2):=(-1)^{\vert v_1\vert \vert v_2\vert}.$ This definition is extended to an arbitrary permutation using a its decomposition into transpositions. 

Recall further that   $\sigma\in S_n$ is called an $(i, n-i)$-unshuffle if it satisfies $\sigma(1)<\dots<\sigma(i)$ and $\sigma(i+1)<\dots<\sigma(n).$ The set of $(i, n-i)$-unshuffles is denoted by $S_{(i, n-i)}.$

\begin{df}\label{li1}(\cite[Def. 5]{KajSta})
An \emph{$L_\infty[1]$-algebra} is a graded vector space $W=\bigoplus_{i\in \ZZ}W_i$ equipped with a collection  ($k\ge1$) of linear maps $m_k \colon \otimes^kW\longrightarrow W$ of degree $1$ satisfying, for every collection of homogeneous elements $v_1, \dots, v_n \in W$:\begin{itemize}
\item[1)] graded  symmetry: for every $\sigma \in S_n$
$$m_n(v_{\sigma(1)}, \dots, v_{\sigma(n)})=\epsilon(\sigma)m_n(v_1,\dots,v_n),$$
\item[2)] relations:  for all $n\ge 1$
$$
\sum_{\substack{i+j=n+1\\i,j\ge 1}}
 \sum_{\sigma\in S_{(i,n-i)}}\epsilon (\sigma)m_j(m_i(v_{\sigma(1)}, \dots,v_{\sigma(i)}),v_{\sigma(i+1)}, \dots, v_{\sigma(n)} )=0.$$
\end{itemize}
In a \emph{curved $L_\infty[1]$-algebra} one additionally allows for an element $m_0\in W_1$ (which can be understood as a   bracket with zero arguments),  one allows $i$ and $j$ to be zero in the relations   2), and one adds the relation corresponding to  $n=0$.   
\end{df}

\begin{rem} 
\label{desuspend}
There is a bijection between \li-algebra structures on a graded vector space $V$ and $\li[1]$-algebra structures on $V[1]$, the graded vector space defined by $(V[1])_i:=V_{i+1}$ \cite[Rem. 2.1]{Vor}. The multibrackets are related by applying the d\'ecalage isomorphisms

\begin{equation}\label{deca}
 (\otimes^n V)[n] \cong \otimes^n(V[1]),\;\; v_1\dots v_n \mapsto v_1\dots v_n\cdot (-1)^{(n-1)|v_{1}|+\dots+2|v_{n-2}|+|v_{n-1}|},
 \end{equation}
 where $|v_i|$ denotes the degree of $v_i\in V$. The bijection extends to the  curved case.
 \end{rem}
 
 From now on, for any $v\in V$, we denote by $v[1]$ the corresponding element in $V[1]$ (which has  degree $|v|-1$). Also, we denote the multibrackets in $\li[1]$-algebras by $\{\cdots\}$, we denote by $d:=m_1$ the unary bracket, and in the curved case we denote $\{\emptyset\}:=m_0$ (the   bracket with zero arguments).

\begin{df} Given an $L_\infty[1]$-algebra $W$, a \emph{Maurer-Cartan element} is a degree $0$ element $\Phi$ satisfying the 
 Maurer-Cartan equation
   \begin{equation}\label{MaCa}
 \sum_{n=1}^{\infty} \frac{1}{n!}\{\underbrace{\Phi,\dots,\Phi}_{n \text{ times}}\}=0.
\end{equation}
(We consider the convergence of this  infinite sum in \S \ref{sec:conv}). We denote by $MC(W)$ the set of Maurer-Cartan elements of $W$. 

If $W$ is a {curved} $L_\infty[1]$-algebra, we define Maurer-Cartan elements by adding $m_0\in W_1$ to the left hand side of eq. \eqref{MaCa} (i.e. by letting the sum in \eqref{MaCa} start at $n=0$).
\end{df}



\subsection{Th. Voronov's constructions of $L_\infty$-algebras as derived brackets}\label{vorpaper}
We recall Th. Voronov's derived bracket construction \cite{Vor}\cite{vor2}, which out of simple data constructs an $\li[1]$-algebra structure. 
 
\begin{df}\label{vdata}
 A \emph{V-data} consists of a quadruple $(L,\ga, P,\Delta)$ where 
 \begin{itemize}
\item  $L$ is a graded Lie algebra (we denote its bracket by $[\cdot,\cdot]$), 
\item $\ga$ an abelian  Lie subalgebra,  
\item $P \colon L \to \ga$ a projection whose kernel is a Lie subalgebra of $L$, 
 \item $\Delta \in Ker(P)_1$ an element such that $[\Delta,\Delta]=0$.
\end{itemize}
When  $\Delta$ is an arbitrary element of $L_1$ instead of $Ker(P)_1$, we refer to  $(L,\ga, P,\Delta)$
as a \emph{curved V-data}. 
\end{df}

\begin{thm}[{\cite[Thm. 1, Cor. 1]{Vor}}]\label{voronovderived}
 Let $(L,\ga, P, \Delta)$ be a curved V-data.
Then $\mathfrak{a}$ is a curved $L_\infty[1]$-algebra for the multibrackets   $\{\emptyset\}:=P\Delta$ and 
($n\ge 1$)
\begin{align}\label{eq:adp}
\{ a_1,\dots,a_n\}&=P[\dots[[\Delta,a_1],a_2],\dots,a_n].
\end{align} 
We obtain a  $L_\infty[1]$-algebra exactly when $\Delta \in Ker(P)$ .
\end{thm}

When $\Delta \in Ker(P)$  there is actually a larger  $L_\infty[1]$-algebra, which contains $\ga$ as in Thm. \ref{voronovderived} as a $L_\infty[1]$-subalgebra.

\begin{thm}[{\cite[Thm. 2]{vor2}}]\label{voronov}
Let $(L,\ga, P, \Delta)$ be a V-data, and denote $D:=[\Delta,\cdot] \colon L \to L$.
Then the space $L[1]\oplus \mathfrak{a}$ is a $L_\infty[1]$-algebra for the differential
\begin{equation}\label{diffV2}
d( x[1], a):= (- (Dx)[1], P(x+Da)),
\end{equation}
the binary bracket
$$\{x[1],y[1]\}=[x,y][1](-1)^{|x|}\in L[1],$$
and for $n\ge 1$:
\begin{align}\label{vorder}
\{ x[1],a_1,\dots,a_n\}&=P[\dots[x,a_1],\dots,a_n]\in \ga,\\
\label{vorderlong}\{a_1,\dots,a_n\}&=P[\dots[Da_1,a_2],\dots,a_n]\in \ga.
\end{align}
Here $x,y\in L$ and $a_1,\dots,a_n \in \mathfrak{a}$. Up to permutation of the entries, all the remaining multibrackets vanish.
\end{thm} 

\begin{nota}
We will denote by $$\ga_{\Delta}^P$$ and by $$(L[1]\oplus \mathfrak{a})_{\Delta}^P$$ the $L_{\infty}[1]$-algebras produced by Thm. \ref{voronovderived} and \ref{voronov}.
We will also often consider the projection
\begin{equation}\label{pphi}
P_{\Phi} :=P \circ e^{[\cdot,\Phi]}\colon L \to \ga.
\end{equation}
\end{nota}

\begin{rem}\label{MCexp}  Let $(L,\ga, P, \Delta)$ be a curved V-data and
  $\Phi\in \ga_0$ as above. Then $\Phi$ is a Maurer-Cartan element of $\ga_{\Delta}^P$ if{f} 
\begin{equation}
P_\Phi \Delta=0,
\end{equation}
or equivalently $\Delta\in ker(P_{\Phi})$. This follows immediately from eq. \eqref{eq:adp}.
\end{rem}


\begin{rem}\label{kerPa}
Let $L'$ be a graded Lie subalgebra of $L$ preserved by $D$ (for example $L'=Ker(P)$). Then $L'[1]\oplus\ga$ is stable under the multibrackets of Thm. \ref{voronov}. We denote by $(L'[1]\oplus \mathfrak{a})_{\Delta}^P$ the induced $\li[1]$-structure. One stresses that it is essential to consider this restriction, since the natural inclusion $Ker(P)[1]\longrightarrow L[1]\oplus A$ is a $L_\infty$-map and  a quasi isomorphism. In particular, if we do not consider this restriction, every solution of the Maurer-Cartan equation is gauge equivalent to an element of $Ker(P)$.
\end{rem}



\subsection{{The main tool}}\label{heart}

The following   statement is the main tool we develop. See \cite[\S 1.3]{FZgeo} for its proof. {It is a statement about Maurer-Cartan elements of $\li[1]$-algebras that arise as in Thm. \ref{voronovderived}. In the applications, these Maurer-Cartan elements will be the objects of interest, since they will correspond to morphisms, subalgebras, etc.}

   \begin{thm}\label{machine}
Let $(L,\ga, P,\Delta)$ be a  filtered V-data and let $\Phi\in MC(\ga_{\Delta}^P)$. Then for all   $\tilde{\Delta}\in L_1$   and  $\tilde{\Phi}\in \ga_0$:
\begin{align*}
\begin{cases}
 [\Delta+\tilde{\Delta},\Delta+\tilde{\Delta}]=0  \\
\Phi+ \tilde{\Phi} \in MC(\ga_{\Delta+\tilde{\Delta}}^{P})\end{cases}
\Leftrightarrow
(\tilde{\Delta}[1],\tilde{\Phi}) \in MC\Big((L[1]\oplus \ga)_{\Delta}^{P_{\Phi}}\Big).
\end{align*} 
In this case, $\ga^P_{\Delta+\tilde{\Delta}}$  is a curved $\li[1]$-algebra. It is a  $\li[1]$-algebra exactly when $\tilde{\Delta}\in Ker(P)$.
\end{thm}

\begin{rem}\label{rem:1)} 
For any   $\tilde{\Phi}\in \ga_0$ we have
 $$\Phi+ \tilde{\Phi} \text{ is a MC element of }\ga_{\Delta}^{P}
\;\;\;\Leftrightarrow \;\;\;
\tilde{\Phi} \text{ is a MC element of }\ga_{\Delta}^{P_{\Phi}}.$$
This is a well-known statement,  
saying that perturbations of a Maurer-Cartan element of $\ga_{\Delta}^{P}$ satisfy themselves a Maurer-Cartan equation, and is a particular case of the  equivalence appearing in Thm. \ref{machine} (obtained setting $\tilde{\Delta}=0$).
\end{rem}

In the special case in which $\Delta=0$ and $\Phi=0$, we obtain the following corollary about the space  of {curved}
 $\li[1]$-algebra structures arising as in Thm. \ref{voronovderived} and Maurer-Cartan elements in there:

\begin{cly}\label{pairMC} Let $L,\ga, P$ such that  
$(L,\ga, P,0)$ is a  filtered V-data.  
The  only non-vanishing  multibrackets of  $(L[1]\oplus \ga)^P_0$, up to permutations of the entries, are
\begin{align*}
d(x[1])&= Px,\\
\{x[1],y[1]\}&=[x,y][1](-1)^{|x|},\\
\{ x[1],a_1,\dots,a_n\}&=P[\dots[x,a_1],\dots,a_n]\;\;\;\;\;\; \text{for all }n\ge 1
\end{align*}
where  $x,y\in L$ and $a_1,\dots,a_n \in \mathfrak{a}$. 

Its Maurer-Cartan elements are characterized by: 
for all  $\tilde{\Delta}\in L_1$  and  $\tilde{\Phi}\in \ga_0$
$$ \begin{cases}
[\tilde{\Delta},\tilde{\Delta}]=0 \\
\tilde{\Phi} \text{ is a MC element of }\ga_{\tilde{\Delta}}^{P} 
\end{cases}
\Leftrightarrow 
(\tilde{\Delta}[1],\tilde{\Phi}) \text{ is a MC element of } (L[1]\oplus \ga)^P_0.
$$ 
\end{cly}

\subsection{Convergence issues}\label{sec:conv}

The left hand side of the Maurer-Cartan equation \eqref{MaCa} is generally an infinite sum.
In this subsection we review Getzler's notion of filtered $\li[1]$-algebra \cite{GetzlerGoe}, which guarantees that the above infinite sum converges. We show that simple assumptions on   V-data ensure that the Maurer-Cartan equations of the (curved) $\li[1]$-algebras we construct in   Thm. \ref{machine} do converge.

\begin{df}
Let $V$ be a graded vector space.  A \emph{complete filtration}  is   a descending filtration by graded subspaces
$$V=\cF^{-1}V\supset\cF^0V\supset \cF^1V\supset \dots$$
such that the canonical projection
$V \to \underset{\leftarrow}{\lim} V/\cF^nV$ {is an isomorphism}.  {Here $$\underset{\leftarrow}{\lim} V/\cF^nV {:=}\{\overset{\rightarrow}{x}\in {\Pi_{n\geq -1}}V/\cF^nV\;:\; P_{i,j}({x_j)=x_i} \text{ when }  i<j\},$$ where $P_{i,j} \colon  V/\cF^jV\longrightarrow V/\cF^iV$ is the canonical {projection} induced by the inclusion ${\cF^jV\subset \cF^i V}$.}
\end{df}

 We define \emph{Maurer-Cartan elements} to be  $\Phi\in W_0\cap \cF^1W$ for which   the left hand side of eq. \eqref{MaCa}   vanishes.

\begin{df}\label{triple}
{Let $(L,\ga, P,\Delta)$ be a curved V-data (Def. \ref{vdata}). 
 We say that this {curved V-data} is }
\emph{filtered} if  
 there exists a complete filtration on $L$ 
such that
\begin{itemize}
\item[a)] The Lie bracket has filtration degree zero, i.e. $[\cF^iL,\cF^jL]\subset \cF^{i+j}L$ for all $i,j \ge -1$,
\item[b)] $\ga_0 \subset \cF^1L$,
\item[c)] the projection $P$ has filtration degree zero, i.e. $P(\cF^iL)\subset 
\cF^iL$ for all $i\ge -1$.
\end{itemize}
\end{df}


 See \cite[\S 1.3]{FZgeo} for the proof of the following proposition.

\begin{prop}\label{prop:fil}
 Let $(L,\ga, P,\Delta)$ be a 
{filtered,}  curved {V-data}.
  Then for every $\Phi\in MC(\ga_{\Delta}^P)\subset \ga_0$:
\begin{itemize} 
\item[1)] the projection $P_{\Phi} :=P \circ e^{[\cdot,\Phi]}\colon L \to \ga$ is  well-defined and has filtration degree zero.
\item[2)]  the curved $\li[1]$-algebra   $\ga_{\Delta}^{P_{\Phi}}$ given by Thm. \ref{voronovderived} is filtered by $\cF^n\ga:=\cF^nL\cap \ga$. Further, the sum on the left hand side of eq. \eqref{MaCa}   converges for any degree zero element $a$ of $\ga$.

\item[3)] if $\Delta \in ker(P)$: the   $\li[1]$-algebra $(L[1]\oplus \mathfrak{a})_{\Delta}^{P_{\Phi}}$ given by Thm. \ref{voronov} is filtered by $\cF^n(L[1]\oplus \mathfrak{a}):=(\cF^nL)[1]\oplus \cF^n\ga$. Further, the sum on the left hand side of eq. \eqref{MaCa}  converges for any degree zero element $(x[1],a)$ of $L[1]\oplus \ga$.
\end{itemize}
\end{prop}

A common way to deal with convergence issues is to  work formally (i.e. in terms of power series in a formal variable $\varepsilon$). {The following is the analogue of Prop. \ref{prop:fil}   in the formal setting:}
  
\begin{lem}\label{lem:eps} Let $(L,\ga, P,\Delta)$ be a curved  V-data {(not necessarily filtered)}.
Let $\Phi\in\ga_0 \otimes \varepsilon\cdot \RR[[\varepsilon]]$.

\item[1)]  for the Maurer-Cartan equation of the curved $\li[1]$-algebra   $(\ga\otimes \RR[[\varepsilon]])_{\Delta}^{P_{\Phi}}$  
the following holds:
the sum on the left hand side of eq. \eqref{MaCa}   converges for any element $a$ of $\ga_0 \otimes \varepsilon\cdot \RR[[\varepsilon]]$.

\item[2)] if $\Delta \in ker(P)$, for the Maurer-Cartan equation of the   $\li[1]$-algebra $\big((L[1]\oplus \mathfrak{a})\otimes  \RR[[\varepsilon]]\big)_{\Delta}^{P_{\Phi}}  $ 
the following holds: 
the sum on the left hand side of eq. \eqref{MaCa}  converges for any element $(x[1],a)\in(L[1]\oplus \ga)_0\otimes \varepsilon\cdot \RR[[\varepsilon]]$.
\end{lem}
\begin{proof}
One checks easily that the  following  is   a curved V-data:
\begin{itemize}
\item the graded Lie algebra $L\otimes \RR[[\varepsilon]]$
\item its abelian subalgebra $\ga \otimes \RR[[\varepsilon]]$
\item the natural projection $P  \colon L\otimes \RR[[\varepsilon]] \to \ga\otimes \RR[[\varepsilon]]$ 
\item $\Delta$,
\end{itemize}  
where the the first three structures   are defined
 by $\RR[[\varepsilon]]$-linear extension.
The natural complete filtration $\{\cF^n\}_{n\ge 0}$ on the vector space $L\otimes  \RR[[\varepsilon]]$ by $\cF^n:=L\otimes \varepsilon^n\RR[[\varepsilon]]$
 satisfies conditions a),  c)  of Def. \ref{triple}. It does not satisfy condition b), however the proof of
 Prop. \ref{prop:fil}, applied to the above  curved V-data, goes through whenever
 $\Phi$ and $a$  lie in  $\ga_0 \otimes \varepsilon\cdot \RR[[\varepsilon]]$.
\end{proof}

   \begin{rem}
  Notice that the curved $\li[1]$-algebra $(\ga\otimes \RR[[\varepsilon]])_{\Delta}^{P}$
  is canonically isomorphic to $(\ga_{\Delta}^P)\otimes \RR[[\varepsilon]]$.
\end{rem} 

\subsection{Equivalences of  Maurer-Cartan elements}\label{sec:sym}
 
Let $W$ be an $\li[1]$-algebra. On  $MC(W)$, the set of Maurer-Cartan elements, there is a canonical  
involutive (singular) distribution $\cD$   which induces an equivalence relation on $MC(W)$ known as \emph{gauge equivalence}. 
More precisely, each $z\in W_{-1}$ defines a vector field $\cY^z$ on $W_0$, whose value at $m\in W_0$ is\footnote{The infinite sum   \eqref{gaugeaction} is guaranteed to converge if $W$ is  filtered and $W_{-1}\subset \cF^{1}W$, see \S \ref{sec:conv}. In the example we consider in this paper, this sum is actually finite, see eq. \eqref{gaugezerodefLA}.}
  
\begin{equation}\label{gaugeaction}
\mathcal{Y}^z|_m:=dz+\{z,m\}+\frac{1}{2!}\{z,m,m\}+\frac{1}{3!}\{z,m,m,m\}+\dots.
\end{equation}
This vector field is tangent to $MC(W)$. The distribution at the point $m\in MC(W)$ is defined as
$\cD|_m=\{\mathcal{Y}^z|_m: z\in W_{-1}\}$.

\begin{rem}\label{rem:d0}
When the differential $d$ vanishes, the Jacobiator of the binary bracket $\{\cdot,\cdot\}$ is zero. Hence $\{\cdot,\cdot\}$
makes the vector space $W_{-1}$ into an ordinary Lie algebra, and  
 the assignment 
 $ W_{-1} \to \chi_0(W_0), z \mapsto (\mathcal{Y}^z)_{lin}:=\{z,\cdot\}\in \chi_0(W_0)$ to the linear part   of  $\mathcal{Y}^z$ is a Lie algebra morphism.
\end{rem}

Consider in particular the $\li[1]$-subalgebra $ker(P)[1]\oplus \ga$ of the $\li[1]$-algebra   of Cor. \ref{pairMC}.   Notice that  the differential vanishes, so Remark \ref{rem:d0} applies.
The vector field associated to a degree $-1$ element $z=(z_L[1],z_\ga)\in ker(P)[1]\oplus \ga$, eval{ua}ted at $m=(m_L[1],m_\ga)\in MC(ker(P)[1]\oplus \ga)$, reads
\begin{align}\label{gaugezerodef}
\mathcal{Y}^z|_m=[z_L,m_L][1]+
\sum_{n\ge 1}\frac{1}{n!}P[[z_L,\underbrace{m_\ga],\dots,m_\ga}_{n \text{ times}}]+
\sum_{n\ge 1}\frac{1}{(n-1)!}P[[[m_L,z_\ga],\underbrace{m_\ga],\dots,m_\ga}_{n-1 \text{ times}}]
\end{align}
where  the square bracket is the graded Lie algebra structure on $L$.\\

We will display explicitly the equivalence relation  induced on {morphisms between Lie algebras in   \S \ref{sec:lamorf}.   {It  turns out that the equivalence classes coincide with the orbits of a 
group action.}


\section{Applications to Lie theory}\label{Lie}

We apply now the machinery developed above to instances in Lie theory. For the examples we treat here, procedures to recover the $\li[1]$-algebras governing  simultaneous deformations
are known \cite{FMY}, but often are not exhibited in explicit form in the literature. Using our machinery, we make the $\li[1]$-algebras structures quite explicit. The results of \S \ref{sec:lamorf} recover a theorem in  \cite{FMY}.
We mention further that the results we obtained in \S \ref{sec:subalg} have been recently extended by Ji from the setting of Lie algebras to that of Lie algebroids \cite{Ji}.


\subsection{Lie algebra morphisms.}\label{sec:lamorf}

Let $(U,[\cdot,\cdot]_U)$ and $(V,[\cdot,\cdot]_V)$ be  finite dimensional Lie algebras. We show that
the   deformations of Lie algebra morphisms $U \to V$ are ruled by a DGLA,
recovering classical results of Nijenhuis and Richardson \cite{NijRic2}, 
and that more generally  the simultaneous deformations of the Lie algebra structures and Lie algebra morphisms  are  ruled by a   $L_{\infty}$-algebra, {recovering  a theorem in  \cite{FMY}}  {by the first author, Markl and Yau}. The set-up of this subsection is a  special case of the one of \S \ref{HomotLiemorph}. We consider the simple instance of Lie algebras separately for the sake of concreteness and clarity of exposition.  {Further, we discuss equivalences}.

We consider the graded manifold $(U\times V)[1]$, and encode the above data as vector fields on this graded manifold. See  \cite[\S 1.4]{FloDiss} or \cite{AlbFlRio} for some basic notions on graded manifolds and the notation; in particular $\chi(U[1])$ denotes the space of vector fields on $U[1]$, and   $\iota \colon U \to \chi_{-1}(U[1])$ identifies elements of $U$ with constant vector fields.  We adopt the following conventions:
\begin{itemize}
\item   The Lie bracket $[\cdot,\cdot]_U$ is encoded by the  vector field $Q_U\in \chi_1(U[1])$ defined by $[[Q_U,\iota_X],\iota_Y]=\iota_{[X,Y]_U}$ for all $X,Y\in U$. {The Jacobi identity for $[\cdot,\cdot]_U$ is equivalent to this vector field being homological (i.e., $[Q_U,Q_U]=0$)}
\item A linear map  $\phi \colon U \to V$ is encoded by $\Phi\in \chi_0((U\times V)[1])$ defined by  
$[\Phi,\iota_{X}]=\iota_{\phi(X)}$ for all $X\in U$. 
\end{itemize} 
  \begin{rem}
We give coordinate expressions for the vector fields $Q_U,Q_V,\Phi$. Choose a basis of $U$, giving rise to coordinates $\{u_i\}$ on $U[1]$, and similarly choosing a basis of $V$ get coordinates $\{v_{\alpha}\}$ on $V[1]$. Then 
\begin{equation}\label{coords}
Q_U=-\frac{1}{2}c_{ij}^k u_i u_j \pd{u_k},\;\;\;\;\;\;Q_V=-\frac{1}{2}d_{\alpha \beta}^{\gamma} v_{\alpha} v_{\beta} \pd{v_{\gamma}},\;\;\;\;\;\;\Phi=-A_{l \eta} u_l \pd{v_{\eta}}
\end{equation}
where $c_{ij}^k $ and $d_{\alpha \beta}^{\gamma}$ are the structural constants of the Lie algebras $U$ and $V$ respectively  and $A_{l \eta}$ is the matrix representing $\phi$ in the chosen basis.
\end{rem}

The map $\phi \colon U \to V$ is a Lie algebra morphism exactly when
\begin{equation}\label{MCQ}
[Q_U,\Phi]+\frac{1}{2}[[Q_V,\Phi],\Phi]=0,
\end{equation}
see for example \cite[p. 176]{NijRic2}.

\begin{lem}\label{keyLA} The following quadruple forms a V-data: 
\begin{itemize}
\item the graded Lie algebra $L:=\chi((U\times V)[1])$
\item its abelian subalgebra $\ga:=C(U[1])\otimes V[1]$
\item the natural projection $P \colon L \to \ga$ with kernel 
\begin{equation*} 
ker(P)=\Big(C(U[1])\otimes C_{\ge 1}(V[1])\otimes V[1]\Big)\; \oplus \;\Big(C(U[1] \times V[1])\otimes U[1]\Big)
\end{equation*}
\item $\Delta:=Q_U +Q_V$,
\end{itemize}
hence by Thm. \ref{voronovderived} we obtain a  $L_{\infty}[1]$-structure  $\ga^P_{\Delta}$.
For every linear map $\phi \colon U \to V$ we have:
$\Phi\in \ga_0$ is a Maurer-Cartan element in  $\ga^P_{\Delta}$ if{f} $\phi$ is a Lie algebra morphism.
\end{lem}
\begin{proof} $Ker(P)$ is a Lie subalgebra of $L$. This can be seen in coordinates, or noticing that the kernel  consists exactly of vector fields on $(U\times V)[1]$ which are tangent to $(U\times \{0\})[1]$.
{Further we have $[\Delta,\Delta]= [Q_U,Q_U]+[Q_V,Q_V]=0$.} Hence the above quadruple forms a V-data. 

The $L_{\infty}[1]$-structure  induced on $\ga$ by  Thm. \ref{voronovderived} is given by the multibrackets
$P [[[Q_U +Q_V,\cdot],\cdots],\cdot]$.
One computes easily in coordinates using \eqref{coords} that $P[Q_V,\cdot]$, $ [[Q_U,\cdot],\cdot]$ and $
[[[Q_V,\cdot],\cdot],\cdot]$ vanish when applied to elements of $\ga$. Hence  only the unary and binary brackets are non-zero, and they are given by 
\begin{align*}
&[Q_U,\cdot]\\
&[[Q_V,\cdot],\cdot]\end{align*}
respectively. Therefore the Maurer-Cartan equation of $\ga^P_{\Delta}$ is given by \eqref{MCQ}. 
\end{proof}

Lemma \ref{keyLA} allows us to apply  Thm. \ref{machine} (and Rem. \ref{rem:1)}). Hence we deduce:
\begin{cly}\label{cor:LA}
Let $U,V$ finite dimensional Lie algebras and  $\phi  \colon U\to V$ a  morphism. Let $(L,\ga,P,\Delta)$ as in Lemma \ref{keyLA}.
\begin{itemize}\item[1)] Let   $\tilde{\phi}\colon U \to V$ be a linear map. Then
$$\phi +\tilde{\phi} \text{ is a Lie algebra morphism }
\;\;\;\Leftrightarrow \;\;\;  {\tilde{\Phi}} \text{ is a MC element of }\ga_{\Delta}^{P_{\Phi}}.$$

\item[2)] 
For all  quadratic vector fields $\tilde{Q}_U$ on $U[1]$ and  $\tilde{Q}_V$ on $V[1]$
and for all linear maps
 $\tilde{\phi}\colon U \to V$: 
 \begin{align*}
  &\begin{cases}
Q_U+\tilde{Q}_U \text{ and } Q_V+\tilde{Q}_V \text{define Lie algebra structures on } U \text{ and }V\\
\phi+ \tilde{\phi} \text{ is a Lie algebra morphism between these new Lie algebra structures}\end{cases}
\\
 \Leftrightarrow
&((\tilde{Q}_U +\tilde{Q}_V)[1],\tilde{\Phi}) \text{ is a MC element of } (L[1]\oplus \ga)_{\Delta}^{P_{\Phi}}.
\end{align*}
\end{itemize}
\end{cly}

\begin{rem}\label{rem:convla}
{We check  that $(L,\ga, P,\Delta)$ is filtered V-data (Def. \ref{triple}), as this is a hypothesis in Thm. \ref{machine}.}
We have a direct sum decomposition $L=\oplus_{k \ge -1}L^k$ where
$L^k:= 
C_{k+1}(U[1])\otimes C(V[1])\otimes U[1] \;\oplus \;  C_{k}(U[1]) \otimes C(V[1])\otimes V[1]$. In other words, $L^k$ is  spanned by monomials in $L$ whose total number of $u$'s and $\pd{v}$'s, in coordinates, is exactly $k+1$.  
Then $\cF^nL:=\oplus_{k\ge n}L^k$ is   a complete filtration of the vector space $L$. 
 One checks easily that $(L,\ga, P,\Delta)$ is filtered V-data.

{An alternative way to check that there are no convergence issues for $e^{[\cdot,\Phi]}$ and the Maurer-Cartan equations appearing in Cor. \ref{cor:LA} is to recall that $U\times V$ is finite dimensional and use a variant of Lemma \ref{Jan2} below.} 
\end{rem}

\subsubsection{{Explicit expressions for the multibrackets}} 
In this subsection we make more explicit the structures of $\ga_{\Delta}^{P_{\Phi}}$ and 
$(L'[1]\oplus \ga)_{\Delta}^{P_{\Phi}}$, where $L'\subset L$ is specified just after Lemma \ref{Jan}.

Given a morphism of Lie algebras $\phi \colon U \to V$, the associated
\emph{Richardson-Nijenhius DGLA}  is given by $\oplus_i \wedge^i U^* \otimes V$, the differential being the Chevalley-Eilenberg differential of $U$ with values in the module $V$ (the module structure is given by $e\in U\mapsto [\phi(e),\cdot]_V$) and the bracket being   the Lie bracket on $V$ combined with the wedge product on $\wedge U^*$ (see \cite[p. 175-6]{NijRic2} or \cite[\S 2.3]{yaelnew}).

 \begin{lem}\label{Jan}
 $\ga_{\Delta}^{P_{\Phi}}$ is the {suspension} of the Richardson-Nijenhius DGLA.
\end{lem}
\begin{proof}  
The $n$-ary bracket of  $\ga_{\Delta}^{P_{\Phi}}$, evaluated on $a_1,\dots,a_n\in \ga$ 
 is  
\begin{eqnarray*}
P_{\Phi}[[[Q_U + Q_V,a_1],\cdots],a_n]\end{eqnarray*}
One computes easily in coordinates that only unary and binary brackets are non-zero, and they are given by 
\begin{align}
\label{RNQ}
P[Q_U + [Q_V,{\Phi}],\cdot]=&[Q_U + [Q_V,{\Phi}],\cdot]\\
\label{RNQ2} P[[Q_V,\cdot],\cdot]=&[[Q_V,\cdot],\cdot].\end{align}
respectively. 
The r.h.s. of \eqref{RNQ} is exactly the Chevalley-Eilenberg differential of the Lie algebra $U$ with values in the module $V$. The r.h.s. of \eqref{RNQ2} is given by the Lie bracket on $V$ combined with  the wedge product on $\wedge U^*$. Hence we obtain the {suspension of the} {Nijenhuis-Richardson} DGLA.
\end{proof}

Up to this point we only looked at deformations of the morphism $\phi \colon U \to V$.
Now we also deform the Lie algebra structures on the vector spaces $U$ and $V$.

Define $L':=\chi(U[1])\oplus \chi(V[1])\subset L$. By Thm. \ref{machine} and Rem. \ref{kerPa} we obtain an  $L_{\infty}[1]$-algebra $(L'[1]\oplus \ga)_{\Delta}^{P_{\Phi}}$, governing the simultaneous deformations of the Lie algebra structures on $U,V$ and of the morphisms. 

\begin{lem}\label{Jan2}
$(L'[1]\oplus \ga)_{\Delta}^{P_{\Phi}}$ has multibrackets of order up to $dim(V)+1$. Its Maurer-Cartan equation is cubic, given by {eq. \eqref{MCLAU}, \eqref{MCLAV}} and \eqref{MCLA} below. 
\end{lem}

\begin{proof}

We write down explicitly the multibrackets of $(L'[1]\oplus \ga)_{\Delta}^{P_{\Phi}}$, as given in Thm. \ref{voronov}.
We denote by $\tilde{Q}_U^i,\tilde{Q}_V^i$ and $\tilde{\Phi}^i$ general (homogeneous) elements of 
 $\chi(U[1]),\chi(V[1])$ and $\ga$ respectively ($i=1,2,\dots$).
The multibrackets involving only $\tilde{\Phi}$ are given exactly by 
\eqref{RNQ} and \eqref{RNQ2} since $\ga_{\Delta}^{P_{\Phi}}$ is a $L_{\infty}$-subalgebra of $(L'[1]\oplus \ga)_{\Delta}^{P_{\Phi}}$. Explicitly, they are 
\begin{align*}
d(\tilde{\Phi})=[Q_U + [Q_V,\Phi],\tilde{\Phi}] \in \ga
\end{align*}
and
\begin{align*}
\{\tilde{\Phi}^1,\tilde{\Phi}^2\}= [[Q_V,\tilde{\Phi}^1],\tilde{\Phi}^2] \in \ga.    
\end{align*}

 Now we compute the multibrackets involving at least one of  $\tilde{Q}_{U}[1]$ or $\tilde{Q}_V[1]$.
For the differential we have in $L[1]\oplus \ga$ :
\begin{align*}
d(\tilde{Q}_U[1])&=-[{Q}_U+{Q}_V,\tilde{Q}_U][1]+ P_{\Phi}(\tilde{Q}_U)=
-[{Q}_U,\tilde{Q}_U][1]+ [\tilde{Q}_U,\Phi] \\
d(\tilde{Q}_V[1])&=-[{Q}_U+Q_V, \tilde{Q}_V][1]+ P_{\Phi}(\tilde{Q}_V)=
-[{Q}_V,\tilde{Q}_V][1]+ 
\frac{1}{k!}[[\dots[\tilde{Q}_V,\underbrace{\Phi],\dots],\Phi}_{k}] 
\end{align*}
where $k=|\tilde{Q}_V|+1$.
For the binary bracket
we have 
\begin{align}
\{(\tilde{Q}^1_U+\tilde{Q}^1_V)[1],(\tilde{Q}^2_U+\tilde{Q}^2_V)[1]\}&=
(-1)^{|\tilde{Q}^1_U+\tilde{Q}^1_V|}\Big([\tilde{Q}^1_U,\tilde{Q}^2_U]+[\tilde{Q}^1_V,\tilde{Q}^2_V]\Big)[1] \in L[1] \nonumber\\
\label{bin2}
\{\tilde{Q}_U[1],\tilde{\Phi}\}&=  P_{\Phi}[\tilde{Q}_U,\tilde{\Phi}]=  [\tilde{Q}_U,\tilde{\Phi}] \in \ga \\
\{\tilde{Q}_V[1],\tilde{\Phi}\}&=  P_{\Phi}[\tilde{Q}_V,\tilde{\Phi}]
\in \ga.\nonumber
\end{align}
From \eqref{bin2} it follows that the only non-zero $n$-brackets with $n\ge 3$ are
\begin{align}
\label{trinary}
\{\tilde{Q}_V[1],\tilde{\Phi}^1,\dots,\tilde{\Phi}^n\}= P_{\Phi}[[\tilde{Q}_V,\tilde{\Phi}^1],\dots,\tilde{\Phi}^n]
\in \ga.
\end{align}

In coordinates it is clear that the operation $[\cdot,\tilde{\Phi}]$ sends 
$C(U[1])\otimes C_i(V[1])\otimes V[1]$ to $C(U[1])\otimes C_{i-1}(V[1])\otimes V[1]$.
As $\tilde{Q}_V\in \chi(V[1]) \cong \sum_{i=1}^{dim(V)}C_i(V[1])\otimes V[1]$, it is clear from eq. \eqref{trinary} that all $n$-brackets vanish for $n> dim(V)+1$.\\

To write down the Maurer-Cartan elements , we can  use eq. \eqref{MaCa} and the formulae for the multibrackets derived above.  Alternatively, by virtue of Cor. \ref{cor:LA}, we know that Maurer-Cartan elements $\tilde{Q}= \tilde{Q}_U[1]+ \tilde{Q}_V[1]+ \tilde{\Phi}$ are characterized by the equations $[Q_U+ \tilde{Q}_U,Q_U+ \tilde{Q}_U]=0$, $[Q_V+ \tilde{Q}_V,Q_V+ \tilde{Q}_V]=0$ and by the equation obtained replacing $Q_U$ by $Q_U+\tilde{Q}_U$ (and similarly for $Q_V$,$\Phi$)
in eq.  \eqref{MCQ} . {The first two equations are equivalent to
\begin{align}\label{MCLAU}
[Q_U,\tilde{Q}_U]+\frac{1}{2}[\tilde{Q}_U,\tilde{Q}_U]=0\\
\label{MCLAV} [Q_V,\tilde{Q}_V]+\frac{1}{2}[\tilde{Q}_V,\tilde{Q}_V]=0
\end{align}
while the third equation reads}
\begin{align}\label{MCLA}
0=&[\tilde{Q}_U,\Phi]+
\frac{1}{2}[[\tilde{Q}_V,\Phi],\Phi]
+[Q_U+[Q_V,\Phi], \tilde{\Phi}]\\
+&[\tilde{Q}_U, \tilde{\Phi}]
+[[\tilde{Q}_V, \tilde{\Phi}],\Phi]
+\frac{1}{2}[[Q_V, \tilde{\Phi}],\tilde{\Phi}] \nonumber\\
+&\frac{1}{2}[[\tilde{Q}_V, \tilde{\Phi}], \tilde{\Phi}]\nonumber.
\end{align}
\end{proof}

\subsubsection{Equivalences of Lie algebras morphisms}

Consider the $\li[1]$-algebra whose Maurer-Cartan elements are pairs of Lie algebra structures and morphisms between them, that is, the $\li[1]$-algebra $\cL:=(L'[1]\oplus \ga)_{\Delta=0}^{P }$ as in  Cor. \ref{pairMC}. Here we discuss the natural equivalence on the set of Maurer-Cartan elements, see \S \ref{sec:sym}.

Elements of $\cL_{-1}$ are of the form
$$z=(z_U[1],z_V[1],z_\ga)\in \;\chi_0(U[1])[1]\;\oplus \; \chi_0(V[1])[1]\;\oplus \;  V[1].$$
Restricting the binary bracket $\{\cdot,\cdot\}_2$ to $\cL_{-1}$   and using the identifications at the beginning of \S \ref{sec:lamorf} we obtain
the ordinary Lie algebra
$$End(U)\times (End(V)\ltimes V)$$ where 
 $End(U)$ and $End(V)$ are endowed with the commutator bracket, $V$ is abelian and    $[A,f]=Af\in V$ for   $A\in End(V)$ and $f\in V$. 

Maurer-Cartan elements  lie in $\cL_0$, so they are of the form 
$$m=(m_U[1],m_V[1],m_\ga)\in \;\chi_1(U[1])[1]\;\oplus \;  \chi_1(V[1])[1]\;\oplus \;  (U[1])^*\otimes V[1],$$ and as described at the beginning of \S \ref{sec:lamorf} their components correspond respectively to a Lie bracket $[\cdot,\cdot]_{m_U}$ on $U$, a Lie bracket $[\cdot,\cdot]_{m_V}$ on $V$, and a  Lie algebra morphism $\phi \colon U\to V$.
 By degree reasons   eq. \eqref{gaugezerodef} reads simply 
\begin{align}\label{gaugezerodefLA}
\mathcal{Y}^z|_m=&\;\;\;\;\;\;[z_U,m_U][1]\;\oplus \;[z_V,m_V][1]\;\oplus \; [z_U+z_V,m_\ga]
+[[m_V,z_\ga],m_\ga] \\
\in &T_{z}\Big(\chi_1(U[1])[1]\;\oplus \;  \chi_1(V[1])[1]\;\oplus \;  (U[1])^*\otimes V[1]\Big) \nonumber. 
\end{align}
The assignment $z \mapsto \mathcal{Y}^z$ vector field is {not} a Lie algebra action:  $z^1=(0,0,z_\ga^1)$ and $z^2=(0,0,z_\ga^2)$ commute, but the vector fields $\mathcal{Y}^{z^1}$ and $\mathcal{Y}^{z^2}$
do not commute.
However restricting suitably we obtain an infinitesimal action, which integrates to the group action of symmetries given in \cite[\S 3]{yaelnew}:
\begin{prop}
The assignment 
$End(U)\times End(V) \to \chi(MC(\cL)), z \mapsto \mathcal{Y}^z$ \emph{is} a Lie algebra morphism. 
It integrates to the group action 
\begin{align*}
\Big(GL(U)\times GL(V) \Big)\times MC(\cL) &\to MC(\cL)\\
(g,h)\;,\;\Big([\cdot,\cdot]_{m_U},[\cdot,\cdot]_{m_U},\phi\Big)& \mapsto \Big(g^*([\cdot,\cdot]_{m_U}), h^*([\cdot,\cdot]_{m_V}), h\circ \phi \circ  g^{-1}\Big).
\end{align*} 
Here the Lie bracket $g^*([\cdot,\cdot]_{m_U})$ is defined as  $g  [g^{-1}\cdot,g^{-1}\cdot]_{m_U}$, and similarly for $h^*([\cdot,\cdot]_{m_V})$. 

The equivalence classes induced by the singular distribution $\cD:=\{\mathcal{Y}^z: z\in \cL_{-1}\}$ on $MC$ agree with the orbits of the this action.
\end{prop}

\begin{proof}
Notice that for $z\in End(U)\times End(V)$ the vector field $\mathcal{Y}^z$ is linear, hence $z\mapsto \mathcal{Y}^z$ is a Lie algebra morphism by Remark \ref{rem:d0}.
We   compute the integral curve  of $\mathcal{Y}^z$ starting at $m=(m_U[1],m_V[1],m_\ga)\in MC(\cL)$.

The first component of $\mathcal{Y}^z$ is $[z_U,\cdot][1]$. Its integral curve starting at $m_U[1]$ is $t\mapsto e^{t[z_U,\cdot]}m_U[1]$, since the latter forms a 1-parameter group and differentiates to $[z_V,\cdot]$ at time zero.
The Lie bracket on $U$ induced by $e^{[z_U,\cdot]}m_U[1]$ is $(exp(z_U))^*([\cdot,\cdot]_{m_U})$ where $exp(z_U)$ is the usual matrix exponential of 
$z_U\in \g\l(U)$ (this follows from the fact that $e^{[z_U,\cdot]}$ is an automorphism of $[\cdot,\cdot]$).
The same argument applies to the second component of $\mathcal{Y}^z$. 

For the third component, the integral curve of $[z_U+z_V,\cdot]$ starting at $m_{\ga}$ is $t\mapsto e^{t[z_U+z_V,\cdot]}m_{\ga}$.
The element $e^{[z_U+z_V,\cdot]}m_{\ga} \in (U[1])^*\otimes V[1]$ corresponds to $exp(z_V)\circ \phi \circ exp(-z_U)\colon U \to V$. This shows that the group action in the statement of this proposition integrates the given Lie algebra action.

{For the last statement we fix $m\in MC(\cL)$ and show that $$\cD_m=\{\mathcal{Y}^z|_m : z=(z_U[1],z_V[1],0)\}.$$ To this aim, just notice that
 $\mathcal{Y}^{(0,0,z_\ga)}|_m=\mathcal{Y}^{(0,[m_V,z_\ga],0)}|_m$
for all $z_\ga\in V[1]$, as a consequence of $[m_V,m_V]=0$.}
\end{proof}

\subsection{Subalgebras of Lie algebras}\label{sec:subalg}

Let $\g$ be a  {finite dimensional} Lie algebra, $U\subset\g$ a Lie subalgebra. We study deformations of the Lie algebra structure on $\g$ and of the subspace $U$ as a Lie subalgebra, similarly to {Richardson}
  \cite{RichSubalg}. 

 {At first,} let $U\subset \g$ be simply a subspace.
 We denote by $Q_{\g}\in \chi(\g[1])$ the homological vector field encoding the Lie algebra structure on $\g$. Choose   a subspace $V$ in $\g$ complementary to $U$. Given a linear map $\phi \colon U \to V$, we view it as an element  $\Phi\in C_1(U[1])\otimes \chi_{-1}(V[1])\subset \chi_0(\g[1])$   defined by  
$[\Phi,\iota_{X}]=\iota_{\phi(X)}$ for all $X\in U$.

\begin{lem}\label{keysubalg}
  The following quadruple forms a curved V-data:
\begin{itemize}
\item the graded Lie algebra $L:=\chi(\g[1])$
\item its abelian subalgebra $\ga:=C(U[1])\otimes V[1]$
\item the natural projection $P \colon L \to \ga$ with kernel 
\begin{equation*}
ker(P)=\Big(C(U[1])\otimes C_{\ge 1}(V[1])\otimes V[1]\Big)\; \oplus \;\Big(C(\g[1])\otimes U[1]\Big)
\end{equation*}
\item $\Delta:=Q_{\g}$,
\end{itemize}
hence by Thm. \ref{voronovderived} we obtain a  curved $L_{\infty}[1]$-structure  $\ga^P_{\Delta}$. 

$\Phi \in \ga_0$ is a MC element in  $\ga_{\Delta}^P$ if{f} $graph(\phi)$ is a Lie subalgebra of $\g$.

Further, the above quadruple forms a V-data if{f} $U$ is a Lie subalgebra of $\g$.
\end{lem}

\begin{proof} To show that the above quadruple forms a curved V-data proceed as in the proof of Lemma \ref{keyLA}.

Rem. \ref{MCexp}  says that $\Phi$ is a Maurer-Cartan element in $\ga_{\Delta}^P$ if{f} 
$
e^{-[\Phi,\cdot]}Q_{\g}\in ker(P).        
$
This condition is equivalent  to asking that for all $X,Y\in U$:
\begin{align*}
\left[\left[e^{-[\Phi,\cdot]}Q_{\g},\iota_X\right],\iota_Y\right]\in U[1]\end{align*}
Using the fact that $e^{-[\Phi,\cdot]}$ is a Lie algebra automorphism of $L$ (to pull it out of the brackets) and that $e^{[\Phi,\cdot]}\iota_X=\iota_X+[\Phi,\iota_X]=\iota_{X+\phi(X)}$, we see that the above is equivalent to
$$[X+\phi(X),Y+\phi(Y)]\in \{Z+\phi(Z): Z\in U\}=graph(\phi),$$
i.e. to $graph(\phi)$ being a Lie subalgebra of $\g$.

The last statement  can be proven  as follows: $Q_{\g}\in ker(P)$
  is equivalent to $[[Q_{\g},\iota_X],\iota_Y]\in U[1]$
  for all $X,Y\in U$, which in turn means that  $U$ is a Lie subalgebra of $\g$. (Alternatively, it follows from the above noticing that   $0$ is a Maurer-Cartan element of $\ga_{\Delta}^P$ if{f} $PQ_{\g}=0$.)
\end{proof}


Lemma \ref{keysubalg}  
 allow us to apply  Thm. \ref{machine}  with $\Phi=0$. 
 We  deduce:
 
 \begin{cly}\label{cor:subalg}
Let $\g$ be a Lie algebra, $U\subset\g$ a Lie subalgebra. Choose    a subspace $V\subset \g$ complementary to $U$, and let
$(L,\ga,P,\Delta)$ {be the V-data} as in Lemma \ref{keysubalg}.

For all $\tilde{Q}_{\g}\in L_1$   
and for all
 linear maps $\tilde{\phi}\colon   U \to V$: 
\begin{align*}
&\begin{cases}
 Q_{\g}+\tilde{Q}_\g   \text{ defines a Lie algebra structure on }\g \\
  graph(\tilde{\phi})  \text{ is a Lie subalgebra of it }
  \end{cases}
  \\
 \Leftrightarrow  
  &(\tilde{Q}_{\g}[1] ,\tilde{\Phi}) \text{ is a MC element of } (L[1]\oplus \ga)_\Delta^P.
\end{align*} 
\end{cly}

\begin{rem}
{The proof   that $(L,\ga, P,\Delta)$ is a  filtered V-data is given in Remark \ref{rem:convla}.}
\end{rem} 
  
 \begin{rem}\label{rem:definc}
By Cor. \ref{cor:subalg}, the Maurer-Cartan elements of     $(L[1]\oplus \ga)_\Delta^P$
 are in bijection  with deformations of the Lie algebra structure on $\g$ 
 and deformations of the subspace $U$ as a Lie subalgebra. 
 
Applying Cor. \ref{cor:LA} to the Lie algebra $U$, to the Lie algebra $\g$ and to the inclusion $i \colon U \hookrightarrow \g$, we obtain an
$L_{\infty}[1]$-algebra whose Maurer-Cartan elements are  deformations of the Lie algebra structure on $\g$ and deformations of $i$
to linear maps $i+\tilde{i} \colon U \rightarrow \g$ whose image is a  Lie subalgebra of the new Lie algebra structure on $\g$. Notice that the two Maurer-Cartan sets are quite different,
as different maps
$i+\tilde{i}$ can have the same image.
\end{rem}

\subsection{Maurer-Cartan elements of  \li-algebra structures}\label{sec:MCli}
 
Fix  a (possibly infinite dimensional) graded vector space $W$.
We show that the space of pairs 
$$\text{($\li[1]$-algebra structures on $W$, Maurer-Cartan elements for this structure)}$$ 
is governed by a Maurer-Cartan equation. {We will ignore all convergence issues in this subsection; they are automatically dealt with if one works formally, see Lemma \ref{lem:eps}.} 
 
{We refer to \cite{AMM} for the background material on coderivations.}  
  Recall that \li[1]-algebra structures on $W$ are in bijection with degree 1 self-commuting coderivations $\Theta$   on  
$\overline{SW}:=\oplus_{k=1}^\infty S^k W$. The {canonical} embedding $\alpha \colon W \hookrightarrow Coder(SW)$, induces a  {canonical} bracket-preserving embedding  $\mathcal{J} \colon Coder(\overline{SW})\hookrightarrow Coder(SW)$ {whose image annihilates $1\in SW$}. One can prove {that all $\li[1]$-algebra structures are obtained by the derived bracket construction:}

\begin{prop}\label{keyli}
Let $W$ be an \li[1]-algebra,  and $ {\Theta}$ the corresponding coderivation of $\overline{SW}$.  
The following quadruple forms a V-data: 
\begin{itemize}
\item the graded Lie algebra $L:=Coder(SW)$
\item its abelian subalgebra $\ga:=\{\alpha_w:w\in W\}$ 
\item the   projection 
$P \colon L \to \ga\;,\;
\tau \mapsto \alpha_{\tau(1)}$
\item $\Delta:=\mathcal{J}\Theta$.
\end{itemize}
The induced $L_{\infty}[1]$-structure on $\ga$ given by   Thm. \ref{voronovderived} is exactly the original \li[1]-structure on $W$, under the canonical identification $W\cong \ga, w \mapsto \alpha_w$. \end{prop}

We apply Cor. \ref{pairMC}, choosing $\Theta=0$ above {and restrict to $\{\tau\in Coder(SW):\tau(1)=0\}=Ker(P)\subset L$ (see Rem. \ref{kerPa})}. 
We obtain:

\begin{cly}\label{cor:mcli}
 $\{\tau\in Coder(SW):\tau(1)=0\}[1]\oplus W$, endowed with the    $\li[1]$-algebra structure specified in 
 Cor. \ref{pairMC}, has the following property:
  for all $\tilde{\Theta} \in Coder(\overline{SW})_1$
  and $\tilde{\Phi}\in W_0$: 
 \begin{align*}
 &\begin{cases}\tilde{\Theta} \text{ defines an $\li[1]$-algebra structure on }W \\
\tilde{\Phi} \text{ is a MC element of this $\li[1]$-algebra structure on }W
\end{cases}\\
 \Leftrightarrow 
&\;\;(\mathcal{J}\tilde{\Theta}[1],\tilde{\Phi}) \text{ is a MC element of } \{\tau\in Coder(SW):\tau(1)=0\}[1]\oplus W
\end{align*}
\end{cly}
One can show that the image of the embedding  $\mathcal{J}$ is exactly $\{\tau\in Coder(SW):\tau(1)=0\}$, so  Cor. \ref{cor:mcli} is a statement about \emph{all} Maurer-Cartan elements of $\{\tau\in Coder(SW):\tau(1)=0\}[1]\oplus W$.

\subsection{$L_{\infty}$-algebra morphisms}
\label{HomotLiemorph}

{We consider deformations of a pair of arbitrary $L_{\infty}[1]$-algebras and of a $L_{\infty}[1]$-morphism between them.
We show that  deformations of the morphism with fixed $L_{\infty}[1]$-algebra structures  are ruled by a $L_{\infty}[1]$-algebra (this follows also from Shoikhet's work, see \cite[\S 3]{BorisGauge}\cite{Hinich:2001}), 
and then show that there is an $L_{\infty}[1]$-algebra governing arbitrary deformations.}

We will use the following notation. When $E$ and $F$ are two vector spaces, we will denote by $L(E,F)$ the set of linear maps from E to F and use $L(E):=L(E,F)$ when $E=F$.

 Let $U$ and $V$ be two graded vector spaces. Denote
$\overline{S(U\oplus V)}:=\oplus_{k\ge 1}S^k(U\oplus V)$.
Let
\begin{equation}
L := L\left(\overline{S(U\oplus V)}, U\oplus V\right) = {\prod_{i\ge 1}\bigoplus_{q+r=i}}L^{q,r}_U\oplus L^{q,r}_V \label{class!},
\end{equation}
where
$$L^{q,r}_U:= \left\{\Pi_U\circ l\circ  \Pi^{q,r}: l\in L(S^{q+r}(U\oplus V), U\oplus V)\right\}$$
for $\Pi^{q,r} \colon  S^{q+r}(U\oplus V) \to S^qU\otimes S^rV$ and 
 $\Pi_U \colon U\oplus V \to U$ the
 canonical projections. 
Consider the subspace $$\mathfrak{a}:={\prod_{q\ge {1}}}  L^{q,0}_V\cong L(\overline{SU},V).$$ Thanks to the decomposition (\ref{class!}) one has a projection $P \colon L \to \mathfrak{a}$. 
Notice that the vector space $L$ has a natural $\ZZ$-grading: $L= \oplus_{n\in \mathbb{Z}}L_n$,  where  a map $l \colon \overline{S(U\oplus V)} \to U\oplus V$ lies in $L_n$ if it raises the degree by $n$.
 
As remarked by Stasheff \cite{StasIntr}, L is a graded Lie algebra: the isomorphism of graded vector spaces  
\begin{equation}\label{eq:a8}
L \cong Coder({\overline{S(U\oplus V))}}
\end{equation}
given in Proposition III.2.1 in \cite{AMM}  allows to define the Lie bracket on $L$, the {\it Nijenhuis-Richardson bracket}, as the pullback  of the graded commutator of coderivations.    
\begin{prop}\label{Vsym}
Let $U$ and $V$ be two graded vector spaces equipped with   $L_\infty[1]$-algebra structures $\mu=(\mu_i)_{i\ge 1}$ and  $\nu=(\nu_j)_{j\ge 1}$,  where $\mu_i\in L^{i,0}_U$ and $\nu_j\in L^{0,j}_V$.
The following quadruple (with the previous notations) forms a V-data: 
\begin{itemize}
\item the graded Lie algebra $L$,
\item its abelian subalgebra $\ga$, 
\item the   projection 
$P \colon L \to \ga$,
\item $\Delta:=\mu +\nu$.
\end{itemize}
\end{prop}

\begin{proof}
To see that $\mathfrak{a}$ is an abelian graded Lie subalgebra of $L$, remark that
 elements of $\mathfrak{a}$ are maps which produce vectors in $V$ and accept only terms in $U$. Therefore their composition is zero. 

Next we show that $KerP$ is a graded Lie subalgebra of L. To this aim use the decomposition $KerP=A\oplus B$ where
 \begin{align*}
A_n&=\bigoplus_{s+r=n, r>0}L^{s,r}_{V[1]}, \\
 B_n&=\bigoplus_{s+r=n}L^{s,r}_{U[1]}.
\end{align*}
Let $\alpha, \alpha'\in A, \beta \in B$ and $\gamma \in Ker P$. One has
$\alpha \circ \beta, \alpha\circ \alpha'  \in A$ and $\beta\circ \gamma \in B$, showing that $KerP=A\oplus B$ is closed under the Nijenhuis-Richardson bracket.
Further since $\nu\in A$ and $\mu \in B$, one has  $\Delta \in Ker P$.

Last we show that  $[\Delta,\Delta]=0$. Indeed,
$$[\Delta,\Delta]=[\mu, \mu] +[\nu, \nu] +2 [\mu, \nu] .$$
Since $\mu$ and $\nu$ are $L_\infty[1]$ algebras,  they can be characterized by the vanishing of $[\mu ,\mu]$ and $[\nu , \nu]$ (see \cite{AMM} section $IV.1$). 
Now, by definition of the bracket, $$[\mu ,\nu]_n(x_1\dots x_n)=\sum_{I\amalg J=[n]}\pm \mu_{\vert J\vert+1} (\nu_{\vert I\vert}(x_I)\cdot x_J)\pm \nu_{\vert J\vert+1} (\mu_{\vert I\vert}(x_I)\cdot x_J)$$ but $\mu$ accepts only terms in $V$, whereas $\nu$ produces elements in $U$, hence the first summand of the right hand side vanishes. Similarly for the second summand. This concludes the proof that 
$(L,\ga,P,\Delta)$ forms a V-data.\end{proof}

\begin{prop} \label{carac}
$\Phi \in MC(\mathfrak{a}^P_\Delta)\Leftrightarrow \Phi$ is a morphism of $L_\infty[1]$-algebras.

\end{prop}

\begin{proof}
Fix $\Phi \in \ga_0$. Our aim is to show that the condition for $\Phi$ to be a Maurer-Cartan element for the  $L_\infty[1]$-algebra $\mathfrak{a}^P_\Delta$ (see Remark \ref{MCexp}), 
$$Pe^{[-,\Phi]}(\mu+\nu)=0,$$  is equivalent to the condition for $\Phi$ to be a morphism of $L_\infty[1]$-algebras, i.e., for all $s\ge 1$  and $u_1,\dots,u_s\in U$:
\begin{equation}
\sum_{I\amalg J=[s]}\Phi_{\vert J\vert +1}(\mu_{\vert I\vert}(U_{I})\cdot U_{J})={\sum_{n=1}^s\frac{1}{n!}}\sum_{I_1\amalg\dots\amalg I_{n}=[s]}\nu_n(\Phi_{\vert I_1\vert}(U_{I_1})\dots \Phi_{\vert I_n\vert}(U_{I_n})),\label{morp}
\end{equation}
where $[s]:=\{1, \dots ,s\}$, $\amalg$ means disjoint union and $U_I=u_{\alpha_1}\dots u_{\alpha_j}$ when $I=\{\alpha_1, \dots, \alpha_j\}.$  Some of the $I_i$'s in the expression $I_1\amalg\dots\amalg I_{n}=[s]$ can be empty. One will use the convention that $\Phi_{\vert \emptyset\vert}(U_{\emptyset})=0$ and $U_I\cdot U_\emptyset=U_I.$
Here we decompose $\Phi$ as a sum of its homogeneous elements with respect to the polynomial degree, i.e.  $\Phi=\sum \Phi_n$ where $\Phi_n \in L^{n,0}_V$.

It will be convenient to   {use the isomorphism \eqref{eq:a8}}
to
view the elements of $L$ as coderivations, because in this case the Lie bracket is the graded commutator. The coderivation corresponding to $\Phi$  (resp. to $\mu$,  $\nu$) will be denoted by $\bar \Phi$  (resp. $\bar \mu$,  $\bar \nu$). An explicit expression is given by (cf prop III.2.1 \cite{AMM})
\begin{equation}
\bar Q (x_1\dots x_n):=\sum_{I\amalg J=[n]} \epsilon_x(I,J) Q_{\vert I\vert }(x_I)\cdot x_J \label{sumtaylor}
\end{equation}
where $I\amalg J$ denotes a disjoint union of {\it non-empty} sets and $\epsilon_x(I,J)$ denotes the sign obtained by applying the koszul sign rule to the action of the permutation $[n]\rightarrow I\amalg J$ on the graded elements $x_i$. In the sequel we will omit this sign, but it is understood to be there unless otherwise stated via a $\pm$ sign.

 $\Phi$ is a Maurer-Cartan element of the $L_\infty[1]$-algebra $\mathfrak{a}^P_\Delta$ if{f}  \begin{equation*} 
Pe^{[-,\bar{\Phi}]}(\bar\mu+\bar\nu)=0.
\end{equation*}
But, with the notation {$ad_{\Phi}:=[-,\Phi]$},  one has $$e^{[-,\bar{\Phi}]}=\sum_{n\ge 0}\frac{1}{n!} {ad_{\bar{\Phi}}}^n,$$ and one can compute  $ {ad_{\bar{\Phi}}}^n(\bar\mu)$ and $ {ad_{\bar{\Phi}}}^n(\bar\nu)$ with the expansion$$ {ad_{\bar{\Phi}}}^n(\tau)=\sum_{k+l=n}(-1)^k{{n \choose k}} \bar{\Phi}^k\tau\bar{\Phi}^{l}.$$
Therefore everything boils down to compute terms of the form $$\bar{\Phi}^k\tau\bar{\Phi}^{l}(u_1\dots u_s).$$
The results of these computations for  $\tau=\bar\nu$ and $\tau=\bar\mu$   with $n=k+l$ are claims \ref{lem1}  and \ref{lem2} respectively, and give the two sides of the equation \eqref{morp}.

\begin{clm}\label{lem1} The term $$pr_V({\bar{\Phi}}^{k}\circ \bar\nu\circ{\bar{\Phi}}^{l}(U_{[s]}))$$ always vanishes except  for $l=n$ for which one has  
$$pr_V({\bar{\Phi}}^{0}\circ \bar\nu\circ{\bar{\Phi}}^{n}(U_{[s]}))=\sum_{I_1\amalg\dots\amalg I_{n}=[s]}\bar\nu_n(\Phi_{\vert I_1\vert}(U_{I_1})\dots \Phi_{\vert I_n\vert}(U_{I_n})).$$
\end{clm}

\begin{clm}\label{lem2}
The term $$pr_V({\bar{\Phi}}^{k}\circ \bar\mu\circ{\bar{\Phi}}^{l}(U_{[s]}))$$ always vanishes, except for  $k=n=1$ for which one has 
$$pr_V({\bar{\Phi}}^{1}\circ \bar\mu(U_{[s]}))=\sum_{I\amalg J=[s]}\Phi_{\vert J\vert +1}(\mu_{\vert I\vert}(U_{I})\cdot U_{J}).$$
\end{clm}
Combining the results of claims \ref{lem1} and \ref{lem2} finishes the proof of Proposition \ref{carac}.
\end{proof}

 We now state a lemma and use it to prove claims \ref{lem1} and \ref{lem2}. 
All along we fix $s\ge 1$ and $u_1,\dots,u_s\in U$.
\begin{lem}\label{technique}For all $t\ge 0$
\begin{equation}
{\bar{\Phi}}^{t}(U_{[s]})=\sum_{I_1\amalg\dots\amalg I_{t+1}=[s]}\Phi_{\vert I_1\vert}(U_{I_1})\dots \Phi_{\vert I_t\vert}(U_{I_t})\cdot U_{I_{t+1}}.\label{phin}
\end{equation}
\end{lem}

\begin{proof}
Apply formula (\ref{sumtaylor}) $t$ times and remark that since $\Phi$ admits only elements in $U$, terms of the form $\Phi(\Phi(U_I)\cdot U_{I'})$ can not appear in the obtained expression. {The case $t=0$ is a convention}.
\end{proof}

\begin{proof}[Proof of claim \ref{lem1}]
We first compute $\bar\nu\circ{\bar{\Phi}}^{l}(U_{[s]})$. We therefore apply the formula (\ref{sumtaylor}) to $\bar\nu$ evaluated on the right hand side of the equation (\ref{phin}), with $t=l$ to get 
$${\sum \nu_{\vert I_{l+1}\vert+j}(\Phi_{\vert I_{\alpha_1}\vert}(U_{I_{\alpha_1}}) \dots  \Phi_{\vert I_{\alpha_j}\vert}(U_{I_{\alpha_j}})\cdot U_{I_{l+1}})\cdot\Phi_{\vert I_{\beta_1}\vert}(U_{I_{\beta_1}}) \dots \Phi_{\vert I_{\beta_k}\vert}(U_{I_{\beta_k}})\cdot U_{I_{l+2}}},$$
where $\{\alpha_1,\dots, \alpha_j\}=J$ and $\{\beta_1,\dots, \beta_k\}=K$, and the sum is over $I_1\amalg\dots\amalg I_{l+2}=[s]$ and $J\amalg K=[l].$

Now, since $\nu$ admits only elements in $U$, the term $U_{I_{l+1}}$ must be absent in the previous expression, i.e. one has
$$\bar\nu\circ{\bar{\Phi}}^{l}(U_{[s]})=\sum \nu_{\vert J\vert}(\Phi_{\vert I_{\alpha_1}\vert}(U_{I_{\alpha_1}})\dots \Phi_{\vert I_{\alpha_j}\vert}(U_{I_{\alpha_j}}))\cdot\Phi_{\vert I_{\beta_1}\vert}(U_{I_{\beta_1}})\dots \Phi_{\vert I_{\beta_k}\vert}(U_{I_{\beta_{k}}})\cdot U_{I_{l+1}},$$
(sum over
$I_1\amalg\dots\amalg I_{l+1}=[s], J\amalg K=[l]).$

We are interested in evaluating the expression ${\bar{\Phi}}^{k}\circ \bar\nu\circ{\bar{\Phi}}^{l}(U_{[s]})$, with $k+l=n$. By applying Lemma \ref{technique} with $t=k$ to the last expression, and by the fact that $\Phi$ admits only terms in $U$, one gets

$${\bar{\Phi}}^{k}\circ\bar \nu\circ{\bar{\Phi}}^{l}(U_{[s]})=\sum \nu_{\vert J\vert}(\Phi_{\vert I_{\alpha_1}\vert}(U_{I_{\alpha_1}})\dots \Phi_{\vert I_{\alpha_j}\vert}(U_{I_{\alpha_j}}))\cdot\Phi_{\vert I_{\beta_1}\vert}(U_{I_{\beta_1}})\dots \Phi_{\vert I_{\beta_k}\vert}(U_{I_{\beta_{k}}})\cdot U_{I_{n+1}}.$$
(sum over $I_1\amalg\dots\amalg I_{n+1}=[s]; J\amalg K=[n]$).

Finally, if one considers the terms in the above formula which belong to $V$, one has
$$pr_V({\bar{\Phi}}^{k}\circ \bar\nu\circ{\bar{\Phi}}^{l}(U_{[s]}))=\sum_{I_1\amalg\dots\amalg I_{n}=[s]}\nu_n(\Phi_{\vert I_1\vert}(U_{I_1})\dots \Phi_{\vert I_n\vert}(U_{I_n})).$$
\end{proof}

\begin{proof}[Proof of claim \ref{lem2}]
We start with evaluating $\bar\mu\circ{\bar{\Phi}}^{l}(U_{[s]})$. We apply the formula (\ref{sumtaylor})  to $\bar\mu$ evaluated on the right hand side of the equation (\ref{phin}), with $t=l$ and remark that  since $\mu$ admits only elements in $U$, terms of the form $\mu(\Phi(U_I)\cdot U_{I'})$ can not appear in the obtained expression. Therefore one has
$$ \bar\mu\circ{\bar{\Phi}}^{l}(U_{[s]})=\sum_{I_1\amalg\dots\amalg I_{l+2}=[s]}\Phi_{\vert I_1\vert}(U_{I_1})\dots \Phi_{\vert I_l\vert}(U_{I_l})\cdot \mu_{\vert I_{l+1}\vert}(U_{I_{l+1}})\cdot U_{I_{l+2}}.$$
We now evaluate ${\bar{\Phi}}^{k}\circ \bar\mu\circ{\bar{\Phi}}^{l}(U_{[s]})$ by applying Lemma \ref{technique} to the previous expression, with $t=k$. Since  $\Phi$ admits only elements in $U$, terms of the form $\Phi(\Phi(U_I)\cdot U_{I'})$ can not appear in the obtained expression. Hence one gets 
(remember that $n=k+l$)
\begin{align*}
&\sum_{I_1\amalg\dots\amalg I_{n+2}=[s]} \pm \Phi_{\vert I_1\vert}(U_{I_1})\dots \Phi_{\vert I_n\vert}(U_{I_n})\cdot \mu_{\vert I_{n+1}\vert}(U_{I_{n+1}})\cdot U_{I_{n+2}} \\ 
+& {\sum_{I_1\amalg\dots\amalg I_{n+2}=[s]} \pm \Phi_{\vert I_1\vert}(U_{I_1})\dots \Phi_{\vert I_{n}\vert+1}( U_{I_{n}}\cdot  \mu_{\vert I_{n+1}\vert}(U_{I_n+1}))\cdot U_{I_{n+2}} }.
\end{align*} 
In the previous expression, there are terms which belong to V only if $n$=$k$=$1$. In this case one has
$$pr_V({\bar{\Phi}}\circ \bar\mu(U_{[s]}))=\sum_{I\amalg J=[s]}\Phi_{\vert J\vert +1}(\mu_{\vert I\vert}(U_{I})\cdot U_{J}).$$
\end{proof}

Prop. \ref{Vsym} {and Prop. \ref{carac}} allow us to apply    Thm. \ref{machine} (and Rem. \ref{rem:1)}) and deduce:
\begin{cly}\label{cor:limorf}
Let $U,V$ be $\li[1]$-algebras and ${\Phi}\in L({ \overline{SU},V)}$ a $\li[1]$-morphism from $U$ to $V$ {and} let $(L,\ga,P,\Delta)$ as in Prop. \ref{Vsym}.
\begin{itemize}
\item[1)]
Let $\tilde{\Phi}\in L_0({ \overline{SU},V)}=\ga_0$.   Then $$\Phi +\tilde{\Phi} \text{ is an $L_{\infty}{[1]}$-morphism }
 \;\;\;\Leftrightarrow \;\;\;
 \tilde{\Phi} \in MC(\ga_{\Delta}^{P_{\Phi}}).$$

\item[2)]
For all degree one coderivations $\tilde{Q}_U$ on $\overline{SU}$ and  $\tilde{Q}_V$ on  $\overline{SV}$  and for all     $\tilde{\Phi} \in L_0({ \overline{SU},V)}$:
\begin{align*}
 &\begin{cases}
Q_U+\tilde{Q}_U \text{ and } Q_V+\tilde{Q}_V \text{ define $L_{\infty}{[1]}$-algebra structures on } U,V \\
 \Phi+ \tilde{\Phi} \text{ is a $L_{\infty}[1]$-morphism  between these $L_{\infty}[1]$-algebra structures} 
 \end{cases}
 \\
 \Leftrightarrow \;\;
 &((\tilde{Q}_U +\tilde{Q}_V)[1], \tilde{\Phi}) \in MC((L[1]\oplus \ga)_{\Delta}^{P_{\Phi}})
\end{align*}
\end{itemize}
\end{cly}

\begin{rem} 
We have a direct product decomposition $L=\prod_{k \ge -1}L^k$ where
$L^k:= L^{k+1,\bullet}_U\oplus L^{k,\bullet}_V$. Here we use the short-hand notation $L^{k,\bullet}_V:=\prod_{r \ge 0}L^{k,r}_V$.  Then $\cF^nL:=\prod_{k\ge n}L^k$ is   a complete filtration of the vector space $L$. One checks easily that $(L,\ga, P,\Delta)$ is filtered V-data (Def. \ref{triple}).
\end{rem}

 \noindent\textbf{Acknowledgements:}  {We thank J. Stasheff for comments, and D. Iacono, M. Manetti, F. Sch\"atz, B. Shoiket, {B. Vallette}, T. Willwacher for useful conversations.}

\noindent M.Z.   thanks Uni.lu for hospitality (Luxembourg, 08/2010, grant FNR/10/AM2c/18). He was partially supported by CMUP (Porto), financed by FCT (programs POCTI, POSI and Ciencia 2007); grants  PTDC/MAT/098770/2008 and 
PTDC/MAT/099880/2008 (Portugal), MICINN RYC-2009-04065 and MTM2009-08166-E (Spain).
 
\noindent Most of the work of Y.F. on this article was done while assistant of Prof. Dr. Martin Schlichenmaier at Uni.lu  (grant R1F105L15), to whom he would like to address his warmest thanks. He benefited form the support of the UAM {through grant MTM2008-02686} (Madrid, 06/2010), and from the MPIM in Bonn (12/2011-01/2012).

\bibliographystyle{habbrv}
\bibliography{DerbibAlg}

\end{document}